\documentclass[graybox]{svmult}

\usepackage{mathptmx}       
\usepackage{helvet}         
\usepackage{courier}        
\usepackage{type1cm}        
\usepackage{makeidx}         
\usepackage{graphicx}        
\usepackage{multicol}        
\usepackage[bottom]{footmisc}

\usepackage{amsmath}
\usepackage{amssymb}

\usepackage[authoryear]{natbib}

\makeindex             
\def\dom{\texttt{Dom}}

\def\part{\texttt{Parts}}
\def\tvar{\texttt{TVar}}
\def\pvar{\texttt{PVar}}
\def\aff{\texttt{Aff}}
\def\ass{\mathcal A}
\def\prec{\texttt{prec}}
\def\post{\texttt{post}}

\newcommand{\tuple}[1]{\vec{#1}}

\newcommand {\indep}[3] {#2 ~\bot_{#1}~ #3}

\newtheorem{theo}{Theorem}[section]

\newtheorem{defin}[theo]{Definition}

\begin{document}

\title*{Dynamic Logics of Imperfect Information: From Teams and Games to Transitions}
\author{Pietro Galliani}
\institute{Pietro Galliani \at University of Sussex, \email{P.Galliani@sussex.ac.uk}}
%
%
\maketitle

\abstract{We introduce a new semantical formalism for logics of imperfect information, based on Game Logic (and, in particular, on van Benthem, Ghosh and Lu's \emph{Concurrent Dynamic Game Logic}). This new kind of semantics combines aspects from game theoretic semantics and from team semantics, and demonstrates how logics of imperfect information can be seen as languages for \emph{reasoning about games}. 
Finally we show that, for a very expressive fragment of our language, a simpler semantics is available.
}

\section{Introduction}
The interplay between \emph{game theoretic semantics} \citep{hintikka83,hintikkasandu97} and \emph{team semantics} \citep{hodges97} is one of the most distinctive phenomena of the field of \emph{logics of imperfect information}.\footnote{In this work, we will not describe in much detail the definitions of the logics of imperfect information nor the motivations which led to their their development. The interested reader who is not familiar with this field of research is referred to \citep{vaananen07b} and \citep{mann11} for a thorough introduction to the area.}

In brief, a game-theoretic semantics defines, for any suitable model $M$ and sentence $\phi$ of its language, a \emph{semantic game} $G^M(\phi)$. Truth and falsity are the defined in terms of properties of this games - usually in terms of the existence of winning strategies for a designated player.\footnote{There are exceptions, however: for example, \citep{sevenster09} and \citep{galliani08} present two, not entirely equivalent, varieties of game-theoretic semantics based on the concept of \emph{Nash Equilibria} for undetermined games.} For example, given a first-order model $M$, the first-order sentence $\exists x \forall y (x = y)$ corresponds to the game in which first the existential player (also called \emph{Eloise}) selects an element $m_x \in \dom(M)$, and afterward the universal player (also called \emph{Abelard}, or \emph{Nature}) selects an element $m_y \in \dom(M)$. The play is then won by Eloise if $m_x = m_y$, and by Abelard otherwise; and, from the fact that Eloise has a winning strategy for this game in any first-order model with at most one element, and Abelard has a winning strategy otherwise, one can conclude that 
\[
	M \models \exists x \forall y (x = y) \Leftrightarrow |\dom(M)| \leq 1.
\]
\emph{Independence-Friendly Logic} (IF-Logic) is among the most studied logics of imperfect information (\citep{hintikka96,tulenheimo09,mann11}). Its language extends the one of first-order logic by adding \emph{slashed quantifiers} $(\exists v / W) \phi(v, \ldots)$, where $W$ is a set of variables: the intended interpretation can be roughly expressed in natural language as ``There exists a $v$, \emph{chosen independently from} $W$, such that $\phi(v, \ldots)$ holds''. This can be represented formally in terms of an additional requirement over the \emph{Skolem function} corresponding to the slashed variable; but, and this is perhaps more in keeping with the informal description of the meaning of a slashed quantifier, the corresponding rule of game-theoretic semantics can be given as
\begin{itemize}
\item If the current subformula is $(\exists x /W) \psi$ and the current assignment is $s$, the existential player picks a value $m \in \dom(M)$. The next subformula is then $\psi$, and the next assignment is $s[m/x]$. 

\emph{The existential player must select the same value for }$x$\emph{ for any two assignments which are indistinguishable modulo }$W$.
\end{itemize}
The second part of this rule is a \emph{uniformity condition} over the strategies of the existential player, and it is the only aspect with respect to which this rule differs from the one for the usual, non-slashed existential quantifier. Because of it, the semantic games corresponding to IF-Logic formulas are, in general, \emph{games of imperfect information} \citep{osborne94}. The class of the semantic games corresponding to IF-Logic formulas is a very natural generalization of the one of those corresponding to First-Order formulas, and \citep{hintikka96} puts forward  a number of convincing arguments as for why these games, and the corresponding logics, may be deserving of investigation. 

However, the usual formulation of this semantical framework is not entirely without drawbacks. Games are complex objects, far more so than variable assignments; and, furthermore, game-theoretic semantics for IF Logic is explicitly \emph{non-compositional} and does not easily lend itself to the study of the properties of open formulas. These issues were one among the motivations for the development and the success of \emph{team semantics} \citep{hodges97}\footnote{Team semantics can also be found mentioned under the names of \emph{Hodges semantics} and of \emph{trump semantics}.}, an equivalent and compositional approach to the semantics of logics of imperfect information. 

In brief, team semantics can be seen as a generalization of Tarski's compositional semantics for First-Order Logic. The fundamental difference between Tarski semantics and team semantics is that, in the latter, satisfaction is not predicated of single assignments but instead of \emph{sets of assignments} (which, after \citep{vaananen07}, will be called \emph{teams} in this work.).

The connection between the game-theoretic semantics and the team semantics for IF-Logic is the following:
\begin{theo}
Let $M$ be a first-order model, let $\phi$ be a suitable formula and let $X$ be a team. Then $X$ satisfies $\phi$ in $M$ if and only if there exists a strategy $\tau$ for Eloise which is winning for her in $G^M(\phi)$ and for all initial assignments $s \in X$.
\end{theo}

Team semantics proved itself an extremely fruitful approach to the study of logics of imperfect information. In particular, its discovery was fundamental for the development of \emph{Dependence Logic}, a logical formalism which is expressively equivalent to IF-Logic\footnote{Strictly speaking, this is the case only with respect to sentences. With respect to open formulas, this is true only if the domain of the team is presumed finite and fixed.} and which separates the notion of dependence/independence from the notion of quantification by eschewing slashed quantifiers in favor of \emph{dependence atoms} $=\!\!(t_1 \ldots t_n)$, whose informal meaning can be stated as ``The value of the term $t_n$ is a function of the values of the terms $t_1 \ldots t_{n-1}$.'' 

Dependence Logic, in turn, was found to be an ideal ground for the discovery of a number of meta-logical results, especially in the area of \emph{finite model theory}.\footnote{We will not attempt to give here a summary of these results; apart from \citep{vaananen07}, the interested reader could refer for example to \citep{kontinenv09,kontinennu09,kontinenv10,kontinen_ja10} and \cite{durand11}. But this brief list far from complete.}

Another direction of research that saw a good amount of recent interest consists in the study of logics of imperfect information obtained by adding to the language of first-order logic atoms corresponding to \emph{non-functional} notions of dependence, such as \emph{Independence Logic} \citep{gradel10}, \emph{Multivalued Dependence Logic} \citep{engstrom10}, and \emph{Inclusion/Exclusion Logic} \citep{galliani12}. The analysis of the relationship between such logics, in particular, is (at least, in the opinion of the author) a promising and, to this moment, largely unexplored subject.\footnote{\citep{galliani12} contains a few basic results along these lines, as does \citep{galliani11b}.}

Despite all of this, team semantics is not entirely without drawbacks either. Its formal definition - formally elegant though it may be - obfuscates somewhat the natural intuitions evident in the underlying game theoretic semantics. Why are the rules for the various connectives of our logics in team semantics the ones that they are? Why do we have \emph{these} connectives, instead of the many others that could be defined in such a framework? Even more importantly, what is the \emph{meaning} of the statement according to which a certain team $X$ satisfies a formula $\phi$, and under which circumstances could one be interested in whether this statement is true or false?

This is of particular relevance for those who, like the author, are interested in the study of extensions and variants of Dependence Logic: without a thorough, exact understanding of the notion of satisfiability in team semantics and of its interpretation, it would not be at all clear whether a certain system of semantic rules - natural though it may appear from a purely formal point of view - holds any logical significance or not.

The present work can be seen as an attempt to find a partial answer to these issues. We will develop a semantics for logics of imperfect information which can be seen as a generalization of \emph{both} game-theoretic semantics and team semantics, and in which formulas have a natural interpretation in terms of either specifications of games or assertions about the properties of games. The resulting formalism can also be seen as an imperfect-information, first order variant of \emph{Game Logic} \citep{pauly03}; and, in fact, it is the hope of the author that the framework described here will prove itself a viable starting point for the establishment of closer relationships between these two research areas.

These notes contain little in terms of new results. This is due to their purpose: rather than proving complex metalogical results about an established formalism, we tried here to present, in as direct a way as we could, a novel way of considering logics of imperfect information and their role.

We finish this introduction by mentioning a number of works on whose results these notes are based. The field of \emph{dynamic semantics}, and in particular \citep{groenendijk91}, was the original source of inspiration for much of the machinery developed here. Furthermore, \citep{bradfield00} provided many of the insights upon which this work is based, and our compositional treatment of games is inspired by \citep{abramsky07}. But from a technical point of view, the formalism which we will present here is a direct descendant of van Benthem, Ghosh and Lu's \emph{Concurrent Dynamic Game Logic} \citep{vanbenthem08b}, and it resembles in particular the ``richer collective modal language'' briefly mentioned in section 4.2. of their insightful paper.

Another work, even more closely related to the present one, is the \emph{Transition Dependence Logic} discussed in Section 3.2 of \citep{galliani2014transition}. However, the semantics which will be discussed in the present paper are more general: indeed, the main purpose of \citep{galliani2014transition} is to illustrate how the usual semantics for Dependence Logic can be interpreted in terms of transitions and assertions about reachability in games of imperfect information, whereas in the present work we will investigate how Game Logic may be extended through notions from Dependence Logic. 

Together, these two works can be considered two attempts -- operating respectively from the ``Game Logic'' side towards the ``Dependence Logic'' one and vice versa -- to highlight the profound similarities and connections between these two largely independent areas of logical investigation and to suggest ways in which either could benefit from the other.
\section{Dynamic Logics of Imperfect Information}
In this section, we will introduce our basic formalism and discuss some of its possible extensions. 
\subsection{Teams, Transitions and Games}
\label{subsect:games}
\begin{defin}[Team variables]
Let $\tvar$ be a nonempty, fixed, not necessarily infinite set of symbols $v_1, v_2, \ldots$ If $v \in \tvar$, we will say that $v$ is a \emph{team variable}.\footnote{Or, simply, a \emph{variable}.}

Let $M$ be a first order model. Then $\ass_M = \dom(M)^{\tvar}$ is the set of all team variable assignments over $M$ with variables in $\tvar$.
\end{defin}
\begin{defin}[Team]
Let $M$ be a first order model. A \emph{team} over it is a subset of $\ass_M$, that is, a set of assignments over $M$.
\end{defin}
\begin{defin}[Transition]
Let $M$ be a first order model. A \emph{transition} on $M$ is a partial function $\tau$ from $\ass_M$ to $\part(\ass_M) \backslash \{\emptyset\}$. 

The domain of $\tau$ is also called its \emph{precondition} $\prec(\tau)$; and its range is also called its \emph{postcondition} $\post(\tau)$. If $X$ is the precondition of $\tau$ and $Y$ is its postcondition, we will also write $\tau:X \rightarrow Y$.
\end{defin}
Note that we require here that, under a well-defined transition, an assignment $s$ in its precondition always has at least one successor $s' \in \tau(s)$.

Another of our main semantic concepts will be the notion of \emph{game}. There exist in the literature of number of different definitions of game, all with their advantages and drawbacks. The usual choice in game-theoretic semantics is to deal with games in \emph{extensive form}, or, more rarely, with games in \emph{strategic form}. Here, however, we will adopt a different option:
\begin{defin}[Games in Transition Set Form]
Let $M$ be a first order model. A first-order game in \emph{transition set form} over $M$ is a nonempty set of transitions on $M$. A \emph{strategy} for the existential player in a game $G = \{\tau_1, \tau_2, \ldots\}$ is simply a transition $\tau_i \in G$.
\end{defin}
A transition for a game describes a \emph{commitment} of our existential players: if $\tau(s) = X$, then the existential player can \emph{ensure} that, if the initial assignment is $s$, the terminal assignment of the game will be in $X$. A choice of a strategy for Eloise in this game, therefore, specifies a set of \emph{conditions} concerning the relationship between the initial and terminal assignments of the game, as well as a \emph{belief state} (the precondition of the transition) under which the strategy is available and these conditions may be imposed. Abelard's strategies may then be seen as \emph{choice functions} $\sigma$ selecting, for any $\tau$ and any assignment $s \in \prec(\tau)$, a specific \emph{outcome} $\sigma(\tau, s) \in \tau(s)$; and hence, the postcondition $\post(\tau)$ of $\tau$ describes the belief state of the existential player about the \emph{outcome} of the game when $\tau$ is played starting from an unknown belief state in $\prec(\tau)$. 

Let us introduce a couple of simple operations between transitions: 
\begin{defin}[Concatenation of Transitions]
Let $\tau$ and $\tau'$ be two transitions with $\post(\tau) = \prec(\tau')$. Then $\tau \circ \tau': \prec(\tau) \rightarrow \post(\tau')$ is the transition defined by 
\[
	(\tau \circ \tau')(s) = \{\tau'(s') : s' \in \tau(s)\}.
\]
\end{defin}
\begin{defin}[Union of Transitions]
Let $\tau_0$ and $\tau_1$ two transitions. Then $\tau_0 \cup \tau_1: \prec(\tau_0) \cup \prec(\tau_1) \rightarrow \post(\tau_0) \cup \post(\tau_1)$ is the transition such that 
\[
	(\tau_0 \cup \tau_1)(s) = \left\{\begin{array}{l l}
			\tau_0(s) & \mbox{ if } s \in \prec(\tau_0) \backslash \prec(\tau_1);\\
			\tau_1(s) & \mbox{ if } s \in \prec(\tau_1) \backslash \prec(\tau_0);\\
			\tau_0(s) \cup \tau_1(s) & \mbox{ if } s \in \prec(\tau_0) \cap \prec(\tau_1)
		\end{array}
	\right.
\]
for all $s \in \prec(\tau_0) \cup \prec(\tau_1)$.
\end{defin}

Using these transitions, we can define a couple of operations over games. 
\begin{defin}[Concatenation of Games]
Let $G$ and $G'$ be games over the same model $M$. Then $G;G'$ is the game containing $\tau \circ \tau'$ for all $\tau \in G$ and all $\tau' \in G'$ with $\post(\tau) = \prec(\tau')$.
\end{defin}
\begin{defin}[Choice between games]
Let $G_0$ and $G_1$ be games over the same model $M$. Then $G \cup G'$ is the game containing $\tau_0 \cup \tau_1$ for every $\tau_0 \in G_0$ and $\tau_1 \in G_1$.
\end{defin}
These operations on games are by no means the only ones that can be studied in this framework. For example, we could consider a \emph{dualization} operator $G^d$, which interchanges the roles of the two players; a \emph{revealing operation} $R(G) = (G^d)^d$, which allows the second player to choose his strategy as a function of the assignment $s$;\footnote{That $R(G) = (G^d)^d$ would then follow at once from the fact that, in our formalism, the second player -- representing the environment -- has no knowledge restrictions.} a \emph{parallel composition} operator $G_1 || G_2$; an \emph{iteration} operator $G^*$; and so on.

Some of these possibilities will be explored later; but first, let us see what we can do with what we already defined.
\subsection{First-Order Concurrent Dynamic Game Logic with Imperfect Information}
\begin{defin}[Parameter Variables]
Let $\pvar = \{p_1, p_2, \ldots\}$ be a fixed, countably infinite set of symbols which is disjoint from $\tvar$ and from the other symbols of our language. We will call it the set of the \emph{parameter variables}.
\end{defin}

\begin{defin}[Syntax]
Let $\Sigma$ be a first-order signature. 

The \emph{game formulas} of our logic are defined as 
\[
	\gamma ::= \epsilon ~|~ \sharp v ~|~ ! v ~|~ \gamma ; \gamma ~|~ \gamma \cup \gamma ~|~ \phi?
\]
where $v$ ranges over $\tvar$ and where $\phi$ is a \emph{belief formula}.

The \emph{belief formulas} of our logic are defined as 
\[
	\phi := \top ~|~ R\tuple t ~|~ \lnot R \tuple t ~|~ t = t' ~|~ t \not = t' ~|~ \sim \phi ~|~ \exists p \phi ~|~ \phi \vee \phi ~|~ \langle \gamma \rangle \phi  
\]
where $R$ ranges over all predicate symbols of our signature, $\tuple t$ ranges over all tuples of terms in our signature of the required lengths and with variables in $\tvar \cup \pvar$, $v$ ranges over $\pvar$ and $\gamma$ ranges over all game formulas.
\end{defin}
We will define the semantics only with respect to formulas with no free parameter variables; the extension to the case in which free parameter variables occur could be done in the straightforward way, by considering \emph{parameter variable assignments}, but for simplicity reasons we will not treat it in this work.
\begin{defin}[Semantics]
Let $M$ be a first order model and let $\gamma$ be a game formula with no free parameter variables. Then $\|\gamma\|_M$ is a game on $M$, defined inductively as follows:
\begin{enumerate}
\item For all teams $X$, $\tau_{\epsilon, X}: X \rightarrow X \in \|\epsilon\|_M$, where $\tau_{\epsilon, X}(s) = \{s\}$ for all $s \in X$;
\item For all teams $X$, all variable symbols $v$ and all functions $F$ from $X$ to $\part(\dom(M)) \backslash \emptyset$, $\tau_{F, X}: X \rightarrow X[F/v] \in \|\sharp v\|_M$, where $\tau(s) = s[F/v] = \{s[m/v] : m \in F(s)\}$ and $X[F/v] = \bigcup \{s[F/v] : s \in X\}$;
\item For all teams $X$ and all variable symbols $v$, $\tau_{\forall v, X}: X \rightarrow X[M/v] \in \|! v\|_M$, where $\tau_{\forall v, X}(s) = s[M/v] = \{s[m/v] : m \in \dom(M)\}$ and\\ $X[M/v] = \bigcup\{s[M/v] : s \in X\}$;
\item If $\tau: X \rightarrow Y \in \|\gamma\|_M$ and $\tau': Y \rightarrow Z \in \|\gamma'\|_M$ then $\tau \circ \tau': X \rightarrow Z \in \|\gamma;\gamma'\|_M$;
\item If $\tau_0: X_0 \rightarrow Y_0 \in \|\gamma_0\|_M$ and $\tau_1: X_1 \rightarrow Y_1 \in \|\gamma_1\|_M$ then $\tau_0 \cup \tau_1: X_0 \cup X_1 \rightarrow Y_0 \cup Y_1 \in \|\gamma_0 \cup \gamma_1\|_M$; 
\item If $X \in \|\phi\|_M$ then $\tau_X : X \rightarrow X \in \|\phi?\|_M$, where $\tau_X(s) = \{s\}$ for all $s \in X$.
\end{enumerate}
If $\phi$ is a belief formula then $\|\phi\|_M$ is instead a set of teams, defined as follows: 
\begin{enumerate}
\item $X \in \|\top\|_M$ for all teams $X$;
\item $X \in \|R\tuple t\|_M$ if and only if $X \subseteq R^M$;
\item $X \in \|\lnot R\tuple t\|_M$ if and only if $X \cap R^M = \emptyset$;
\item $X \in \|t = t'\|_M$ if and only if $t\langle s\rangle = t'\langle s\rangle$ for all $s \in X$;
\item $X \in \|t \not = t'\|_M$ if and only if $t\langle s\rangle \not = t'\langle s\rangle$ for all $s \in X$;
\item $X \in \|\sim \phi\|_M$ if and only if $X \not \in \|\phi\|_M$; 
\item $X \in \|\exists p \phi\|_M$ if and only if there exists an element $m \in \dom(M)$ such that $X \in \|\phi[m/p]\|_M$;\footnote{Or, to be more formal, if and only if there exists an unused constant symbol $c$ and an element $m \in \dom(M)$ such that $X \in \|\phi[c/p]\|_{M(c \mapsto m)}$.}
\item $X \in \|\phi_1 \vee \phi_2\|_M$ if and only if $X \in \|\phi_1\|_M$ or $X \in \|\phi_2\|_M$; 
\item $X \in \|\langle \gamma\rangle \phi\|_M$ if and only if there exists a team $Y$ and a $\tau: X \rightarrow Y \in \|\gamma\|_M$ such that $Y \in \|\phi\|_M$.
\end{enumerate}

If $\tau \in \|\gamma\|_M$, we will write $M \models_\tau \gamma$ and will say that $\tau$ is \emph{a strategy of} $\gamma$; and if $X \in \|\phi\|_M$, we will write $M \models_X \phi$ and we will say that $X$ \emph{satisfies} $\gamma$.

If $M \models_X \phi$ for all $X \subseteq \ass_M$, we will say that $\phi$ is \emph{true} in $M$, and we will write $M \models \phi$; and finally, if $M \models \phi$ for all first order models $M$ we will say that $\phi$ is \emph{valid}.
\end{defin}

As usual, we will write $\bot$ for $\lnot \top$, $\phi \wedge \psi$ for $\sim (\sim \phi \wedge \sim \psi)$, $\phi \rightarrow \psi$ for $\sim \phi \vee \psi$, $\phi \leftrightarrow \psi$ for $(\phi \rightarrow \psi) \wedge (\psi \rightarrow \phi)$, $\forall v \phi$ for $\lnot (\exists v \lnot \phi)$ and $[\gamma]\phi$ for $\sim \langle \gamma\rangle \sim \phi$. 

The intuition behind this logical system should be clear. Game formulas describe games: for example, $\sharp v$ is the game in which the existential player picks new values for the variables $v$, $\gamma_1 \cup \gamma_2$ is the game in which the existential player chooses whether to play $\gamma_1$ or $\gamma_2$, and so on. Belief formulas describe instead conditions over teams, that is, over \emph{belief sets}; and the connection between game formulas and belief formulas is given by the \emph{test} operation $\phi?$, which corresponds to the game that merely verifies whether the initial belief state of the existential player satisfies $\phi$, and by the modal operator $\langle \gamma \rangle \phi$, which states that the existential player can play $\gamma$ and reach a final belief state in which $\phi$ holds.\footnote{In other words, in $\langle \gamma\rangle \phi$ the belief formula $\phi$ specifies a \emph{winning condition} for the game formula $\gamma$.}
\subsection{More Constructors and Atoms}
Now that we have defined our basic framework, let us examine a few ways to extend it with some of the connectives and predicates typically studied in the framework of logics of imperfect information.
\subsubsection{Tensor}
\label{subsubsect:tensor}
The semantic rule for the disjunction of formulas that we have in our semantics is the one corresponding to the \emph{classical} disjunction: a team $X$ satisfies $\phi \vee \psi$ if and only if it satisfies $\phi$ or it satisfies $\psi$. However, in the field of logics of imperfect information there exists another, perhaps more natural form of disjunction, which arises directly from the game theoretical interpretation of the disjunction in classical logic. Following \citep{vaananen07b}, we will write it as $\phi \otimes \psi$.

Its satisfaction condition is defined as follows:
\begin{itemize}
 \item For any model $M$, team $X$ and pair of (belief) formulas $\phi$ and $\psi$, $M \models_X \phi \otimes \psi$ if and only if there exist  $Y$ and $Z$ such that $M \models_Y \phi$, $M \models_Z \psi$, and $X = Y \cup Z$.
\end{itemize}

Can we model this connective in our framework? Certainly! Indeed, we have that $\phi \otimes \psi$ is equivalent to 
\[
\langle \phi? \cup \psi?\rangle \top.
\]
This can be verified simply by unraveling our definitions: indeed, $M \models_{\tau'} \phi?$ if and only if $\tau':Y \rightarrow Y$ is such that $M \models_Y \phi$ and $\tau'$ is the identity on $Y$, $M \models_{\tau''} \psi?$ if and only if $\tau'':Z \rightarrow Z$ is such that $M \models_Z \psi$ and $\tau''$ is the identity on $Z$, and hence $\tau : X \rightarrow X$ satisfies $\phi? \cup \psi?$ if and only if it is the identity on $X$ and $X$ can be split into two subteams which satisfy $\phi$ and $\psi$ respectively.

Using the tensor operator, we can define inequalities of tuples of terms: more precisely, if $\tuple t_1$ and $\tuple t_2$ are of the same length $n$ we can define $\tuple t_1 \not = \tuple t_2$ as 
\[
\bigotimes_{i=1}^n (\tuple t_{1i} \not = \tuple t_{2i})
\]
It is then easy to see that $M \models_X \tuple t_1 \not = \tuple t_2$ if and only it $\tuple t_1$ and $\tuple t_2$ differ for all $s \in X$.
\subsubsection{The Announcement Operator}
In \citep{galliani10b}, the \emph{announcement operator} $\delta t$ was introduced and its properties were studied. For any model $M$, team $X$ and term $t$ of the signature of $M$, its satisfaction rule can be stated as
\begin{itemize}
\item $M \models_X \delta t \phi$ if and only if for all $m \in \dom(M)$, $M \models_{X_{|t = m}} \phi$, where 
\[
	X_{|t=m} = \{s \in X : t\langle s\rangle = m\}
\]
\end{itemize}
The reason why this operator is called an \emph{announcement operator} should be clear: if a team $X$ represents a belief state of an agent, asserting that $M \models_X \delta t \phi$ is equivalent to stating that if the ``true'' value of the term $t$ is announced, then updating the belief set with this new information will lead to a state in which $X$ is true. 

In our framework, this operator can be simulated through quantification and the tensor: indeed, if $p$ is a parameter variable which does not occur in $t$ or in $\phi$, it is trivial to see that $\delta t \phi$ is equivalent to $\forall p (p \not = t \otimes (p = t \wedge \phi))$. 
\subsubsection{Dependencies}
As we wrote in the introduction, \emph{Dependence Logic} extends the language of first-order logic with \emph{dependence atoms} $=\!\!(t_1 \ldots t_n)$, with the intended meaning of ``The value of $t_n$ is a function of the values of $t_1 \ldots t_{n-1}$.'' Formally, this is expressed by the following satisfaction condition: 
\begin{itemize}
	\item For any first order model $M$, every $n \in \mathbb N$ and every $n$-uple of terms $t_1 \ldots t_n$, $M \models_X =\!\!(t_1 \ldots t_n)$ if and only if every $s, s' \in X$ which assign the same values to $t_1 \ldots t_{n-1}$ also assign the same value to $t_n$.
\end{itemize}
These atoms can be easily represented in our formalism. 

First, let us consider the case of the \emph{constancy atoms} $=\!\!(t)$, which - according to the above condition - are satisfied in a team $X$ if and only if the value of $t$ is the same for all the assignments in $X$. 

Clearly, this atom can be defined in terms of our existential quantifier: more precisely, if $p$ is a variable which does not occur in $t$ then $=\!\!(t)$ is equivalent to $\exists p (t = p)$.

Furthermore, we can decompose the dependence atom into announcement operators and constancy atoms: indeed, as mentioned in \citep{galliani10b}, $=\!\!(t_1 \ldots t_n)$ is equivalent to $\delta t_1 \ldots \delta_{t_{n-1}} =\!\!(t_n)$. Therefore all dependence atoms are expressible in our formalism.

Some recent work examined the logics obtained by adding to first-order logic atoms corresponding to non-functional notion of dependencies. We will examine three such dependencies here, and see how all of them can be defined using the language of our logic: 
\begin{description}
\item[Inclusion Atoms (\citep{galliani12}):] For all tuples of terms $\tuple t_1$ and $\tuple t_2$, of the same length, all models $M$ and all teams $X$, $M \models_X \tuple t_1 \subseteq \tuple t_2$ if and only if for all $s \in X$ there exists a $s' \in X$ with $\tuple t_1\langle s\rangle = \tuple t_2\langle s'\rangle$. Let $n = |\tuple t_1| = |\tuple t_2|$, and let $\tuple p$ be a $n$-uple of parameter variables not occurring in $\tuple t_1$ and in $\tuple t_2$: then $\tuple t_1 \subseteq \tuple t_2$ can be seen to be equivalent to 
\[
	\forall \tuple p(\tuple p \not = \tuple t_2 \rightarrow \tuple p \not = \tuple t_1).
\]
\item[Exclusion Atoms (\citep{galliani12}):] For all tuples $\tuple t_1$ and $\tuple t_2$ of the same length, all models $M$ and all teams $X$, $M \models_X \tuple t_1 ~|~ \tuple t_2$ if and only if $\tuple t_1\langle s\rangle \not = \tuple t_2\langle s'\rangle$ for all $s, s' \in X$. Hence, a team $X$ satisfies $\tuple t_1 ~|~ \tuple t_2$ if and only if it satisfies 
\[
	\forall \tuple p (\tuple p \not = \tuple t_1 \vee \tuple p \not = \tuple t_2)
\]
where, once again, $\tuple p$ is a tuple of fresh parameter variables.
\item[Independence Atoms (\citep{gradel10}):] Let $\tuple t_1$, $\tuple t_2$ and $\tuple t_2$ be three tuples of terms, not necessarily of the same length. Then, for all models $M$ over a suitable signature and for all teams $X$, $M \models_X \indep{\tuple t_1}{\tuple t_2}{\tuple t_3}$ if and only if for all $s, s' \in X$ with $\tuple t_1 \langle s \rangle = \tuple t_1\langle s'\rangle$ there exists a $s'' \in X$ such that $\tuple t_1 \tuple t_2\langle s\rangle = \tuple t_1 \tuple t_2\langle s''\rangle$ and $\tuple t_1 \tuple t_3\langle s'\rangle = \tuple t_1 \tuple t_3\langle s''\rangle$.

If $\tuple p_2$ and $\tuple p_3$ are disjoint tuples of fresh parameter variables of the required arities then we can express this condition as
\[
	\delta \tuple t_1 \forall \tuple p_2 \forall \tuple p_3 (\tuple t_2 \tuple t_3 \not = \tuple p_2 \tuple p_3 \rightarrow (\tuple t_2 \not = \tuple p_2 \vee \tuple t_3 \not = \tuple p_3)).
\]
\end{description}
\subsubsection{Variable Hiding}
As mentioned in the introduction, IF-Logic adds to the language of first-order logic \emph{slashed quantifiers} $(\exists v / W) \phi$, where $W$ is a set of formulas, with the intended meaning of ``there exists a $v$, \emph{chosen independently from} $v$, such that $\phi$''. 

In our framework, it is possible to detach this notion of \emph{variable hiding} from the very notion of quantification as follows: 
\begin{defin}[Independent Transitions]
Let $M$ be a first order model, let $W \subseteq \tvar$ be a set of team variables, and let $\tau: X \rightarrow Y$ be a transition. We say that $\tau$ is \emph{independent on} $W$ if and only if 
\[
	s \equiv_W s' \Rightarrow \tau(s) = \tau(s')
\]
for all $s, s' \in X$, where $s \equiv_W s'$ if and only if $s(v) = s'(v)$ for all $v \in \tvar \backslash W$.
\end{defin}
\begin{defin}[Variable Hiding]
Let $G = \{\tau_1, \tau_2, \ldots\}$ be a game in transition form over a first order model $M$, and let $W$ be a set of variables. Then $(G / W)$ is the game defined as 
\[
	(G / W) = \{\tau_i : \tau_i \in G, \tau_i \emph{ independent on } W\}.
\]
\end{defin}
Hence, for every game formula $\gamma$ and for every $W$ we can now define the game formula $\gamma / W$, whose interpretation in a model $M$ is given by 
\[
	G^M(\gamma / W) = (G^M(\gamma) / W).
\]
In the terms in which we formulated our semantics, this corresponds to the following rule: 
\begin{itemize}
\item For all $M$, $\tau$ and $\gamma$,  $\tau \in \|\gamma / W\|_M$ if and only if $\tau \in \|\gamma\|_M$ and $\tau$ is independent on $W$.
\end{itemize}
It is trivial to see that this definition gives to the game formula $\sharp v/W$ the same interpretation of the IF Logic quantifier $\exists v / W$.
\subsection{Iteration}
Another of the operations typically considered in Game Logic is \emph{iteration}: in brief, given a game $\gamma$, the game $\gamma^*$ is the one in which $\gamma$ is played zero, one or more times, and the existential player chooses when to exit the ``loop''.

This can be described in our framework as follows: 
\begin{itemize}
\item For all $M$, all game formulas $\gamma$ and all transitions $\tau: X \rightarrow Y$, $M \models_\tau \gamma$ if and only if there exists a $n \in \mathbb N$ and strategies $\tau_0 \ldots \tau_{n-1}$ such that $M \models_{\tau_i} \gamma$ for all $i$ and $\tau = \tau_{\epsilon, X} \circ \tau_0 \circ \ldots \circ \tau_{n-1}$.
\end{itemize}
\subsection{Intuitionistic Implication}
The \emph{intuitionistic implication} was defined in \citep{abramsky08} as follows: 
\begin{itemize}
\item For all $M$ and all belief formulas $\phi_1$ and $\phi_2$, $M \models_X \phi_1 \hookrightarrow \phi_2$ if and only if for every $Y \subseteq X$ such that $M \models_Y \phi_1$ it holds that $M \models_Y \phi_2$.
\end{itemize}
The properties of this operator were then studied in \citep{yang10}, where it was shown that the expressive power of dependence logic augmented with this operator is expressively equivalent to full second order logic. 

Nothing, in principle, prevents us from adding this operation directly to our language. 

However, it is perhaps more interesting to consider an idea mentioned (not in relation to intuitionistic implication) in \citep{vanbenthem08b} and define first the \emph{inclusion operator} $\langle \subseteq\rangle \phi$, where $\phi$ is a belief formula, whose satisfaction condition is given by 
\begin{itemize}
\item For all models $M$ and teams $X$, $M \models_X \langle \subseteq \rangle \phi$ if and only if there exists a $Y \subseteq X$ such that $M \models_Y \phi$. 
\end{itemize}
As usual, we can define $[\subseteq] \phi$ as $\sim( \langle \subseteq\rangle (\sim \phi))$; and it is now easy to see that $\phi_1 \hookrightarrow \phi_2$ is equivalent to $[\subseteq](\phi_1 \rightarrow \phi_2)$.
\subsection{Parallel Composition}
As the last connective for this work, we will now consider the \emph{parallel composition}. 

A naive implementation of such an operator stumbles into a small, quite obvious problem: how should we deal with parallel games in which both ``branches'' modify the same variable? Many possible answers, some of which rather sophisticated, have been considered in other contexts; here, however, we will favor the straightforward, if somewhat brutal, option of requiring that no variable is modified in both branches.
\begin{defin}[Affected Variables]
Let $\gamma$ be any game formula. The set $\aff(\gamma)$of its \emph{affected variables} is defined inductively as follows: 
\begin{enumerate}
\item $\aff(\epsilon) = \emptyset$; 
\item For all belief formulas $\phi$, $\aff(\phi?) = \emptyset$;
\item For all team variables $v$, $\aff(\sharp v) = \aff(!v) = \{v\}$;
\item For all game formulas $\gamma$ and $\gamma'$, $\aff(\gamma; \gamma') = \aff(\gamma \cup \gamma') = \aff(\gamma || \gamma') = \aff(\gamma) \cup \aff(\gamma')$; 
\item For all sets of variables $W$, $\aff(\gamma / W) = \aff(\gamma)$.
\end{enumerate}
\end{defin}
Given this definition, we can add the parallel composition to our syntax: 
\begin{itemize}
\item If $\gamma_1$ and $\gamma_2$ are two game formulas with $\aff(\gamma_1) \cap \aff(\gamma_2) = \emptyset$, then $\gamma_1 || \gamma_2$ is a game formula.
\end{itemize}
But how to define its semantics? 

As for the cases of sequential composition and union, it will be useful to first define the parallel composition of two transitions. 
\begin{defin}[Parallel Composition of Transitions]
Let $\tau_0: X \rightarrow Y_0$ and $\tau_1: X \rightarrow Y_1$ be two transitions, and let $\tuple v_0$, $\tuple v_1$ be two disjoint tuples of variables. Then $\tau_0(\tuple v_0||\tuple v_1) \tau_1$ is the transition given by 
\begin{align*}
	(\tau_0(\tuple v_0||\tuple v_1) \tau_1)(s) &= \{s[\tuple m_0/\tuple v_0][\tuple m_1 / \tuple v_1] : \exists s_0 \in \tau_0(s),\\
	&s_2 \in \tau_1(s) \mbox{ s.t. } m_0 = s_0(\tuple v_0) \mbox{ and } m_1 = s_1(\tuple v_1)\}.
\end{align*}
\end{defin}

At this point, defining the semantic rule for parallel composition is trivial: 
\begin{itemize}
\item For all $M$, for all $\gamma_0$ and $\gamma_1$ with $\aff(\gamma_0) \cap \aff(\gamma_1) = \emptyset$, $M \models_{\tau} \gamma_0 ||\gamma_1$ if and only if there exist $\tau_0$ and $\tau_1$ such that $M \models_{\tau_0}\gamma_0$, $M \models_{\tau_1}\gamma_1$ and $\tau = \tau_0 (\aff(\gamma_0) || \aff(\gamma_1)) \tau_1$.
\end{itemize}
In principle, nothing would prevent us from also defining a $\gamma_0 (\tuple v_0 || \tuple v_1) \gamma_1$ operator, along similar lines: such an operator would correspond to playing $\gamma_0$ and $\gamma_1$ in parallel, and at the end of the game updating the assignment according to $\gamma_0$ for the variables in $\tuple v_0$ and according to $\gamma_1$ for the variables in $\tuple v_1$.

It is worth observing that the postcondition of a parallel composition cannot be inferred by the postconditions of its components. 

For example, let $s_0$ and $s_1$ be two assignments with $s_0(v) = s_0(w) = 0$ and $s_1(v) = s_1(w) = 1$, and consider the transitions 
\begin{align*}
&\tau: \{s_0, s_1\} \rightarrow \{s_0, s_1\}, \tau(s_0) = \{s_0\}, \tau(s_1) = \{s_1\};\\
&\tau': \{s_0, s_1\} \rightarrow \{s_0, s_1\}, \tau(s_0) = \{s_1\}, \tau(s_1) = \{s_0\}.
\end{align*}
Now, $\tau$ and $\tau'$ have the same precondition and postcondition. However, it is easy to see that $\post(\tau(v|w)\tau) = \{s_0, s_1\}$, whereas $\post(\tau(v|w)\tau' = \{s_0[1/w], s_1[0/w]\}$. 

Because of this phenomenon, parallel composition - as well as the general version of our variable hiding operator - will not be treatable in the simpler semantics that we will develop in the next section.
\section{Transition Semantics}
In the semantics that we have developed so far, the interpretation of a game formula consists in the set of the transitions it allows -- or, equivalently, in the set of all the strategies available to the existential player in the corresponding game. 

It is natural, at this point, to question whether it is \emph{necessary} to carry all this information in our rules. After all, the only way in which the interpretation of a belief formula may depend on the interpretation of a game formula is if the belief formula contains a subexpression of the form $\langle \gamma \rangle \phi$; and, in this case, all that is relevant to our interpretation is the pre- and postconditions of the transitions for $\gamma$, and not the details of which initial assignments can go to which sets of final assignments.

Therefore, it is natural to consider the following, alternative semantics for our game formulas: 
\begin{defin}[Transition Semantics]
\label{defin:trasem}
Let $M$ be a first order model, let $X$ and $Y$ be teams, and let $\gamma$ be a game formula over the signature of $M$. Then we write $M \models_{X \rightarrow Y} \gamma$, and we say that $X \rightarrow Y$ is an \emph{admissible transition} for $\gamma$, if and only if there exists a $\tau \in \|\gamma\|_M$ with $\prec(\tau) = X$ and $\post(\tau) = Y$.
\end{defin}

Of course, this definition makes sense as a semantics only if game connectives are compatible with it, in the sense that the pre- and post- conditions of the strategies for a composed game are a function of the pre- and post- conditions of the strategies for its components. As we already saw, the parallel composition operator does not satisfy this condition: hence, we will exclude it from the analysis of this section.

Furthermore, the slashing operator is also incompatible with Definition \ref{defin:trasem}. Indeed, let $s_0$ and $s_1$ be two assignments with $s_0(x) = 0$, $s_1(x) = 1$, and $s_0 \equiv_x s_1$, and let $G_0 = \{\tau\}$ and $G_1 = \{\tau'\}$, where 
\begin{itemize}
\item $\tau: \{s_0, s_1\} \rightarrow \{s_0, s_1\}$, $\tau(s_0) = \{s_0, s_1\}$, $\tau(s_1) = \{s_0, s_1\}$; 
\item $\tau': \{s_0, s_1\} \rightarrow \{s_0, s_1\}$, $\tau'(s_0) = \{s_0\}$, $\tau'(s_1) = \{s_0\}$; 
\end{itemize}
From the point of view of pre- and postconditions, these two games are absolutely identical: indeed, both of them accept only $\{s_0, s_1\}$ as an initial belief state, and return it as the only possible output belief state.

However, $G_0/\{x\} = G_0$ while $G_1/W$ is empty.

For this reason, we will consider here the following subset of our language: 
\begin{defin}[Transition Logic]
\emph{Transition Logic} is the sublanguage of our formalism in which 
\begin{enumerate}
\item The parallel composition operator does not occur; 
\item The variable hiding operator may only be applied to game formulas of the form $\sharp v$.
\end{enumerate}
\end{defin}
This language is very expressive: in particular, it is easily seen to be as expressive as \emph{Team Logic} \citep{vaananen07b}, and hence, by \citep{kontinennu09}, as full Second Order Logic.

And for this sublanguage, transition semantics is indeed compositional.
\begin{theo}[Rules for Transition Semantics]
Let $M$ be a first order model and let $X$ and $Y$ be teams. Then
\begin{enumerate}
\item $M \models_{X \rightarrow Y} \epsilon$ if and only if $X = Y$; 
\item $M \models_{X \rightarrow Y} \sharp v$ if and only if there exists a $F$ such that $Y = X[F/v]$;
\item $M \models_{X \rightarrow Y} \sharp v / W$ if and only if there exists a $F$, independent on $W$, such that $Y = X[F/v]$; 
\item $M \models_{X \rightarrow Y} ~! v$ if and only if $Y = X[M/v]$; 
\item $M \models_{X \rightarrow Y} \gamma_1;\gamma_2$ if and only if there exists a $Z$ such that $M \models_{X \rightarrow Z} \gamma_1$ and $M \models_{Z \rightarrow Y} \gamma_2$;
\item $M \models_{X \rightarrow Y} \gamma_1 \cup \gamma_2$ if and only if there exist $X_1$, $X_2$, $Y_1$ and $Y_2$ such that $M \models_{X_1 \rightarrow Y_1} \gamma_1$, $M \models_{X_2 \rightarrow Y_2} \gamma_2$, $X_1 \cup X_2 = X$ and $Y_1 \cup Y_2 = Y$;
\item $M \models_{X \rightarrow Y} \phi?$ if and only if $M \models_X \phi$ and $X = Y$;
\item $M \models_{X \rightarrow Y} \gamma^*$ if and only if there exists a $n \in \mathbb N$ and $Z_1, Z_2, \ldots Z_n$ such that 
\begin{itemize}
\item $Z_1 = X$; 
\item $Z_n = Y$; 
\item For all $i=1 \ldots n-1$, $M \models_{Z_i \rightarrow Z_{i+1}} \gamma$.
\end{itemize}
\end{enumerate}

Furthermore, let $M$ be a first order model, let $\gamma$ be a game formula over it and let $\phi$ be a belief formula over it. Then, for all teams $X$, $M \models_X \langle \gamma\rangle \phi$ if and only if there exists a $Y$ such that $M \models_{X \rightarrow Y} \gamma$ and $M \models_Y \phi$.
\end{theo}
\begin{proof}
None of the cases poses any difficulty whatsoever. As an example, we show the case of the sequential composition operator.
\begin{description}
\item[$\Rightarrow$:] Suppose that $M \models_{X \rightarrow Y} \gamma_1;\gamma_2$. Then, by definition, there exists a $\tau: X \rightarrow Y$ such that $M \models_\tau \gamma_1; \gamma_2$. But this can be the case only if there exists a $Z$, a $\tau_1: X \rightarrow Z$ and a $\tau_2:Z \rightarrow Y$ such that $M \models_{\tau_1} \gamma_1$, $M \models_{\tau_2} \gamma_2$ and $\tau = \tau_1 \circ \tau_2$. Therefore, $M \models_{X \rightarrow Z} \gamma_1$ and $M \models_{Z \rightarrow Y} \gamma_2$, as required.
\item[$\Leftarrow$:] Suppose that $M \models_{X \rightarrow Z} \gamma_1$ and $M \models_{Z \rightarrow Y} \gamma_2$. Then there exist two transitions $\tau_1: X \rightarrow Z$ and $\tau_2:Z \rightarrow Y$ such that $M \models_{\tau_1} \gamma_1$ and $M \models_{\tau_2}\gamma_2$. Hence, $M \models_{\tau_1 \circ \tau_2} \gamma_1;\gamma_2$; and finally, $M \models_{X \rightarrow Y} \gamma_1;\gamma_2$, as required.
\end{description}
\end{proof}
This framework can be extended in many different ways: for example, we could easily add new operations, such as generalized quantifiers after \citep{engstrom10}, or \emph{atomic games} from our signature, or \emph{generalized atomic formulas} as in \citep{kuusisto11}. Or we could add even more game operations, such as for example a \emph{adversaral choice} $\gamma_1 \cap \gamma_2$,\footnote{It is not difficult to see that the transition semantics for this connective would be: $M \models_{X \rightarrow Y} \gamma_1 \cap \gamma_2$ if and only if there exist $Y_1$ and $Y_2$ such that $Y_1 \cup Y_2 = Y$, $M \models_{X \rightarrow Y_1} \gamma_1$ and $M \models_{X \rightarrow Y_2} \gamma_2$.} or we could consider a multi-agent framework as in \citep{abramsky07}, or we could consider the equilibrium semantics-based variant of this formalism, or so on: in general, it appears that much of what has been done in the field of logics of imperfect information can, at least in principle, be transferred to this formalism.

In conclusion, it is the hope of the author that the above described \emph{transition semantics} may provide an useful unifying framework for a number of distinct contributions to the filed, as well as a contribution to the exploration of the relationship between logics of imperfect information and logics of games. 
\begin{acknowledgement}
The author wishes to thank Jouko V\"a\"an\"anen for a number of useful suggestions and comments about previous versions of this work. Furthermore, the author thankfully acknowledges the support of the EUROCORES LogICCC LINT programme.
\end{acknowledgement}
\bibliographystyle{plainnat}
\bibliography{biblio}
\end{document}